\documentclass[12pt]{amsart}

\usepackage{amscd}
\usepackage{amssymb,amsmath,amsthm}

\usepackage{amsrefs}
\usepackage{mathrsfs}
\usepackage{enumerate}
\usepackage[top=1in, bottom=1in, left=1.25in, right=1.25in]{geometry}

\theoremstyle{plain}
\newtheorem{thm}{Theorem}[section]
\newtheorem{lem}[thm]{Lemma}
\newtheorem{cor}[thm]{Corollary}
\newtheorem{defn-lem}[thm]{Definition-Lemma}

\newtheorem{prop}[thm]{Proposition}

\theoremstyle{definition}

\newtheorem{rem}[thm]{Remark}

\def\md #1#2#3#4#5 {\left(
                        \begin{matrix}
             #1 & #2 \\
             #3 & #4
                        \end{matrix}
                      \right)- #5}

\def\ma #1#2#3#4 {\left(
                        \begin{matrix}
             #1 & #2 \\
             #3 & #4
                        \end{matrix}
                      \right)}

\def\End{\operatorname{End}}
\def\Ad{\operatorname{Ad}}

\def\Tr{\operatorname{Tr}}
\def\GL{\operatorname{GL}}

\newcommand{\mc}{\mathcal}
\newcommand{\lb}{\langle}
\newcommand{\rb}{\rangle}
\newcommand{\mb}{\mathbb}
\begin{document}
\title [On a gauge action on sigma model solitons]
       {On a gauge action on sigma model solitons}

\begin{abstract}
In this paper  we consider a gauge action on sigma model solitons over noncommutative tori as source spaces, with a target space made of two points introduced in \cite{DKL:Sigma}. Using new classes of solitons from Gabor frames, we quantify the condition about how to gauge a Gaussian to a prescribed Gabor frame.
\end{abstract}

\author{Hyun Ho \, Lee}

\address {Department of Mathematics\\
          University of Ulsan\\
         Ulsan, South Korea 44610 }
\email{hadamard@ulsan.ac.kr}

\keywords{Nonlinear sigma model, Gauge action, Gabor frames, Noncommutative torus, Noncommutative solitons}

\subjclass[2000]{Primary:58B20, 35C08. Secondary:58B16, 58J05, 42B35}
\date{}
\thanks{This research was supported by Basic Science Research Program through the National Research Foundation of Korea(NRF) funded by the Ministry of Education(NRF-2015R1D1A1A01057489).}
\maketitle

\section{Introduction}
Noncommutative analogues of non-linear $\sigma$-models appeared in \cite {DKL:Sigma, DKL:Sigma2} for the first time.  Later other examples including noncommutative spacetimes were considered by \cite{MR:Sigma, Lee:Sigma}. Among them  there is a continuous analog of the Ising model which consists of field maps from a noncommutative torus to a two-point space. Cosidering only an energy term in the action or excluding a gravity related term in the action, the stable maps are called noncommutative harmonic maps and in this particular case such maps correspond to some smooth projections in the noncommutative torus. It turned out that enegy minimizing ones carry a nontrivial topological charge and satisfy a Belavin-Polyakov bound \cite{L:Harmonic}. 

The construction of such maps depends on a geometric picture or a strong Morita equivalence. Since we can view a noncommutative torus $A_{\theta}$ as an endomorphism algebra of a suitable finitely generated projective bundle, we can think of a projection in $A_{\theta}$ as an operator on the bundle. The bundle  is in fact a bimodule over two different noncommutative tori with operator valued Hermitian structures compatible with each other. By choosing suitable vectors $\xi$ in the module, we consider Rieffel-type projections (see the paragraph after Proposition \ref{P:standard}) and lift the field equation on the noncommutative torus to an equation of $\xi$ on the module.  The vectors in the module both inducing Rieffel-type projections and satisfying a (anti) self duality equation are called noncommutative instantons or solitons following G. Landi.         

Using the idea of a natural transformation on the bundle a gauge action on noncommutative solitons is defined by the right multiplication of invertible elements $g$ of a different noncommutative torus $A_{\alpha}$.  This gauge action is well behaved with Rieffel-type projections so that the vector $\xi \cdot g$ generates a Rieffel-type projection again, and satisfy a (anti) self duality equation whenever $\xi$ does. An important class of $\xi$'s, that are Gaussians, is already known to be solutions for the self duality equation with a constant parameter and the condition when two Gaussians to be gauged each other is characterized in \cite{DKL:Sigma, DKL:Sigma2}. However, for the generic case it is not true that any solution vector $\xi$ could be gauged away to a Gaussian solution  since there is an obstruction in the form of $\overline{\partial}$-equation (see Corollary \ref{C:Gaugeaction}). Nonetheless, there is a good chance that a class of vectors $\xi$ could be gauged to a Gaussian and it is our purpose to provide an affirmative example for this question.  In this direction it is necessary to know noncommutative solitons other than Gaussians and recently new classes of sigma-model solitons are discovered by  Dabrowski, Landi, and Luef \cite{DLF:Sigma} under the observation that a problem in a time-frequency analysis and Gabor analysis is equivalent to find Rieffel-type projections in a noncommutative torus.   

This paper is organized as follows; In section \ref{S:sigmamodelontorus}, we explain a nonlinear $\sigma$-model on noncommutative tori introduced by Dabrowski, Krajewski, and Landi and define Rieffel-type projections using Hilbert module frames.  In section \ref{S:Gabor}, we introduce a class of functions called Gabor frames whose name is originated from Gabor analysis and clarify the condition for a vector $\xi$ to be a Gabor frame in terms of Hermitian structures on the module. Then, in Section \ref{S:Solitons}  we show that  a Gaussian could be gauged to a hyperbolic secant based on Theorem  \ref{T:Trace} which provides a useful tool to check whether a concrete Gabor frame is gauged to a Gaussian with Corollary \ref{C:Gaugeaction} in one hand.

\section{$\sigma$-model on noncommutative torus} \label{S:sigmamodelontorus}
Let $\mc{A_{\theta}}$ be a $*$-algebra consisting of power series of the form
\[a=\sum_{(m,n) \in \mathbb{Z}^2} a_{mn}U_1^mU_2^n\]
with $a_{mn}$ a complex-valued Schwarz function on $\mathbb{Z}^2$, or decreasing rapidly. Two unitaries $U_1, U_2$ have a commutation relation
\begin{equation}\label{E:NC}
U_2U_1=e^{2\pi i \theta}U_1U_2.
\end{equation}
For $\theta$ irrational, there is a unique faithful trace $\Tr$ on $A$ given by
\[\Tr(\sum_{m,n}a_{mn}U_1^mU_2^n)=a_{00}.\]

One can equip $\mc{A_{\theta}}$ with a norm $\|a\|=\sup_{(m,n) \in \mathbb{Z}^2} |a_{mn}| < \infty$ and the closure of $\mc{A_{\theta}}$ with respect to this norm is the universal $C\sp*$-algebra  $A_{\theta}$ generated by two unitaries satisfying the relation (\ref{E:NC}): $\mc{A_{\theta}}$ is dense in $A_{\theta}$ and is a pre-$C\sp*$-algebra.  Also, it is well-known that $\mc{A_{\theta}}$ is the smooth subalgebra of $A_{\theta}$, and closed under the holomorphic functional calculus \cite{CO:NCG}. Throughout the article, we are interested in the case $\theta$ irrational and call both $A_{\theta}$ and $\mc{A_{\theta}}$ noncommutative torus without confusion. 

To define the noncommutative action functional for morphisms from a pre-$C^*$ algebra $B$ to a pre-$C^*$-algebra $A$,  note that there is a formal prescription due to Mathai and Rosenberg \cite{MR:Sigma}; recall that a spectral triple $(A, H, D)$ is given by an involutive $*$-algebra represented as bounded operators on a Hilbert space $H$ and a self-adjoint (unbounded) operator $D$ with a compact resolvant such that commutators $[D, a]$ are bounded for all $a \in A$. A spectral triple $(A, H,D)$ is said to be $\emph{even}$ if the Hilbert space  $H$ is endowed with a super-grading $\gamma$ that commutes with all $a\in A$ and anti-commutes with $D$. In addition,  we say a spectral triple $(A, H, D)$ is $(2,\infty)$-summable if $\Tr_{\omega}(a|D|^{-2}) < \infty$ where $\Tr_{\omega}$ is the Dixmier trace. With a $(2, \infty)$-summable even spectral triple  $(A,H,D)$  one can define a positive Hochschild 2-cocycle $\psi_2$ given by
\[ \psi_2(a_0, a_1,a_2)= \frac{i}{2 \pi}\Tr_{\omega}((1+\gamma)a_0[a_1,D][a_2,D]|D|^{-2}).\]
We can compose it with a field map $\phi:B\to A$ where $B$ is a target space and $A$ represents a string worldsheet in  a noncommutative formalism of the classical $\sigma$-model. To assign a number to any field map we evaluate the induced cocycle on a suitably chosen element of $B\otimes B\otimes B$. Such an element is taken as the noncommutative analogue of the metric on the target, and  we choose a positive element of the form 
\[ G= \sum_i b_0idb_i^idb_2^i\] 
in the space of universal $2$-forms $\Omega^2(B)$. Then the quantity 
\[S(\phi)=\phi^*(\psi)(G) \ge 0\] is the action functional of non-linear $\sigma$-model in noncommutative geometry. 

There is a well-known even spectral triple $(\mc{A_{\theta}}, H, D)$ for the noncommutative torus $\mc{A_{\theta}}$ with 
\[ \gamma=\ma 1 0 0-1 ,  \, D=\partial_1 \sigma_1+\partial_2\sigma_2 \]
where $\sigma_1,\sigma_2$ are Pauli matrices (see Section \ref{S:Solitons} for derivations $\partial_i$'s). When the target space is a two-point space, then $B$ is just two dimensional complex vector space $\mathbb{C}^2$. Since a morphism $\phi: \mathbb{C}^2 \to \mc{A_{\theta}}$ is determined by the image of $e$ the characteristic function on a point, we denote it by a projection  $ p\in \mc{A_{\theta}}$.  Taking  $G=de de\in \Omega^2(\mathcal{B})$ the action functional can be written as
\begin{equation}\label{E:Action}
 S(p)=\Tr(\partial p \overline{\partial}p),
\end{equation}
where $\partial$ or $\overline{\partial}$ are the derivations coming from the complex structure on  noncommutative torus(see Section \ref{S:Solitons}). It is known from  \cite{DKL:Sigma, DKL:Sigma2} or \cite{Lee:Sigma} that the Euler-Lagrange equation for this functional is
\begin{equation}\label{E:EL}
p(\Delta p) -(\Delta p) p=0
\end{equation}
where $\Delta$ is the Laplacian. 

It is well known that there exist a lot of projections in $A_{\theta}$, which is of real rank zero \cite{EE:Irrational},  contrary to the fact that a noncommutative torus is a deformation quantization of commutative two torus. But it is unclear whether there are smooth projections in $\mc{A_{\theta}}$. Thus it was a remarkable discovery of M. Rieffel to construct a projection in $\mc{A_{\theta}}$ \cite{R:NonTorus}, so that a morphism $\phi:\mb{C}^2 \to \mc{A_{\theta}}$ is well-defined. In fact, there is a systematic way to construct projections in a $*$-algebra with a left action module and the dual action algebra. Accordingly we call such projections Rieffel-type projections; for the moment, $A$ is a $*$-algebra. Suppose that there is a $*$-algebra $B$ that is strongly Morita equivalent to A via the bimodule $\Xi$ (see Section \ref{S:Gabor} for the definition).
If we denote the $A$($B$)-valued hermitian inner product by ${}_{A}\lb \, , \,\rb (\lb \, ,\, \rb_B) $, then ${}_{A}\lb \xi, \xi \rb$ is a projection in $A$ provided that $\lb \xi, \xi \rb_B=1_B$. More generally, if we have $\xi_1, \xi_2,\dots, \xi_n$ in $\Xi$, then the matrix, whose $i,j$ entry is ${}_{A}\lb \xi_i, \xi_j\rb $, is a projection in $M_n(A)$ provided that   $\sum_{k=1}^{n}\lb \xi_k, \xi_k \rb_B=1_B$. We call the set $\{ \xi_1, \dots, \xi_n\}$ a module frame for $\Xi$. More precisely, it is called a \emph{(Parseval) standard module frame} $\{ \xi_1, \dots, \xi_n\}$ for $\Xi$ \cite{FL:Frame}. In general, a standard module frame for $\Xi$ is a set $\{ \xi_1, \dots, \xi_n\}$ such that
 \begin{equation}\label{E:Framecondition}
 c_1 {}_A \lb \xi, \xi \rb \le \sum_{i} {}_A\lb \xi, \xi_i \rb {}_A\lb \xi_i, \xi \rb \le c_2 {}_A\lb \xi, \xi \rb, \quad \text{for all $\xi \in \Xi$,}
 \end{equation}
  for positive constants $c_1$ and $c_2$. Since we are interested in a projection in $A$ rather in $M_{n}(A)$ the matrix algebra of $A$, we restrict ourselves to the case of a single frame $\{\eta\}$ for $\Xi$.
The following is one of strategies to find a single standard module frame which was used by many experts.
\begin{prop}\label{P:standard}
Let $\eta$ be an element such that $\lb \eta, \eta \rb_B$ is invertible. Then $\{\eta\}$ is a standard module frame.
\end{prop}
\begin{proof}
 Suppose that a positive element $\lb \eta, \eta \rb$ is invertible, then its spectrum is bounded below for a positive number $c_1 >0$ and bounded above by the norm of it, say $c_2$. Thus, by the functional calculus, $c_1 1_B \le \lb \eta, \eta \rb \le c_2 1_B$. Since \[ \begin{split}
  {}_A\lb \xi, \eta \rb {}_A\lb \eta, \xi \rb&={}_A\lb {}_A\lb \xi, \eta \rb \cdot \eta, \xi \rb\\
                                             &={}_A\lb \xi \cdot \lb \eta, \eta \rb_B, \xi \rb,
\end{split}\] (\ref{E:Framecondition}) is satisfied. Therefore the invertibility of $\lb \eta, \eta \rb_B$ is a sufficient condition for $\eta$ to be a standard module frame.
\end{proof}

  Once we have a frame $\eta$ as in Proposition \ref{P:standard}, we get a Parseval one $\tilde{\eta}$ by the normalization and obtain a  projection ${}_A\langle \tilde{\eta}, \tilde{\eta} \rangle$ in $A$. We call such a projection Rieffel-type projection and the first example of Rieffel-type projections in $\mc{A_{\theta}}$ was found by M. Rieffel using  a compactly supported smooth function as $\eta$ \cite{R:NonTorus}, but later F. Boca discovered another one using $\eta$, a Gaussian (A hard computation involving a quantum theta function was needed to show that $\{ \eta \}$ is a standard frame)\cite{B:QT}. Recently, F. Luef noticed that a fundamental duality principle in Gabor analysis is linked to the invertibility of $\lb \eta, \eta \rb_B$ in the case of noncommutative torus and found a large class of standard module frames which include previous examples of  Rieffel and Boca \cite{L:Gabor}. Surprisingly this class of standard module frames gives rise to minimizing solutions of (\ref{E:EL}) as claimed in \cite{DLF:Sigma}. We  are going to explain this fact more carefully and give a detailed proof in Section \ref{S:Solitons}.

\section{Gabor frames and Noncommutative tori}\label{S:Gabor}
In this section, we summarize the strong Morita equivalence of $\mc{A_{\theta}}$ with its dual $B$ in terms of Gabor analysis from  \cite{L:Gabor},\cite{L:VBoverNT}, and explain a class of functions generates Rieffel-type projections in $\mc{A_{\theta}}$.

We say that two pre $C\sp*$-algebras $A$ and $B$ are strongly Morita equivalent if there is  a bimodule $\Xi$, on which both $A$ and $B$ act left and right respectively, equipped with  $A$-valued inner product ${}_{A}\lb \, , \, \rb$ and $B$-valued inner product $\lb \, ,\, \rb_B$ which satisfy the following conditions; for any $f, g \in \Xi$, and $a\in A, b\in B$
\begin{align*}
{}_{A}\lb f, g \rb^*={}_{A}\lb g, f \rb &, \quad \lb f, g\rb_B^{*}= \lb g, f \rb_B, \\
{}_{A}\lb a \cdot f, g \rb= a \cdot {}_{A}\lb f,g \rb &, \quad \lb f, g\cdot b \rb_B=\lb f, g \rb_B \cdot b, \\
f\cdot \lb g, h\rb_B &= {}_{A}\lb f, g\rb \cdot h,\\
(a\cdot f) \cdot b &= a\cdot (f \cdot b).
\end{align*}
 
Let $\pi:(x,\omega)\in \mathbb{R}^2  \to B(L^2(\mathbb{R}))$ be a (projective) representation defined by

\[(\pi(x, \omega)\xi) (t)= e^{2\pi i t \omega}\xi(t-x)\], or
\[\pi(x, \omega)=M_{\omega}T_{x}\] where $(M_{\omega}\xi )(t)=e^{2\pi i t \omega }\xi(t)$ and $(T_x \xi)(t)=\xi(t-x)$.
Then  the canonical commutation relation for $M_{\omega}$ and $T_x$  holds,
\begin{equation}\label{E:CCR}
M_{\omega}T_x=e^{2\pi i x\omega}T_xM_{\omega}.
\end{equation}
It follows that
\[\pi(z)\pi(z')=e^{-2\pi i x \cdot\eta} \pi(z+z') \text{\, for $z=(x, \omega)$, $z'=(y,\eta)$}.\]
We can easily check $c:\mathbb{R}\times \mathbb{R} \to \mathbb{T}$  defined by $c(z,z')=e^{-2\pi i x \eta}$ for $z=(x, \omega)$, $z'=(y,\eta)$ is a 2-cocycle.

Let $\Lambda$ be a lattice of $\mathbb{R}^2$( for our purpose, we may assume that $\Lambda$ is of the form $\theta\mathbb{Z}\times \mathbb{Z}$). Then  $\mathcal{G}(g, \Lambda )=\{ \pi(\lambda)g \mid \lambda \in \Lambda \}$ in $L^{2}(\mathbb{R})$ is said to be a Gabor system.  A Gabor system is a Gabor frame for $L^2(\mathbb{R})$ if there exist $\alpha, \beta >0$ such that
\begin{equation}\label{E:Gabor}
\alpha\| f \|_2^2 \le \sum_{\lambda \in \Lambda}| \lb f, \pi(\lambda)g \rb|^2 \le \beta \|f\|_2^2.
\end{equation}
In this case, $g$ is called a Gabor atom or window in time-frequency analysis, but we call it a Gabor frame abusing notation since $g$ will give rise to a module frame in our setting.\\
Recall that $l^1(\Lambda, c)$ is a $l^1(\Lambda)$ with a twisted convolution of $\mathbf{a}$ and $\mathbf{b}$ defined by
\[   \mathbf{a} \natural \mathbf{b}(\lambda)= \sum_{\mu\in \Lambda } a(\mu)b(\lambda-\mu)c(\mu, \lambda-\mu),\]
and involution $\mathbf{a}^*=(a^*(\lambda))$ of $\mathbf{a}$ is given by
\[a^*(\lambda)=\overline{c(\lambda, \lambda)a(-\lambda)} \,\, \text{for $\lambda \in \Lambda$.}\]
Then $C^*(\Lambda, c)$ is the completion of $l^1(\Lambda, c)$ under $\pi$. More precisely,  $C^*(\Lambda, c)$ is the completion of the involutive representation of $\mathbf{a}$'s,   \[\pi(\mathbf{a})=\sum_{\lambda \in \Lambda} a(\lambda)\pi(\lambda) \,\, \text{for $\mathbf{a}=(a(\lambda))_{\lambda \in \Lambda}$} \]
with the product
\[\pi(\mathbf{a})\pi(\mathbf{b})=\pi(\mathbf{a}\natural \mathbf{b}),\]
and the involution
\[ \pi(\mathbf{a})^*= \pi (\mathbf{a}^*).\]
 Weighted analogues of the twisted group algebra are needed to obtain $A$ in terms of Gabor analysis ; for $s\ge 0$ let $l^1_s(\Lambda)$ be the space of all sequences $\mathbf{a}$ with $\| \mathbf{a} \|_{s}=\sum_{\lambda \in \Lambda} |a(\lambda)|(1+|\lambda |^2)^{s/2}$. We consider $(l_s^1(\Lambda), \natural, \sp)$ and the involutive representation of it, so
\[ \mathcal{A}^1_s (\Lambda, c)=\{ T\in B(L^2(\mathbb{R}))\mid T=\sum_{\lambda \in \Lambda} a(\lambda)\pi(\lambda), \, \|\mathbf{a}\|_s < \infty\}\]
is an involutive algebra with respect to the norm $\displaystyle \| T \|=\sum_{\lambda \in \Lambda} |a(\lambda)| (1+| \lambda|^2)^{s/2}$. It turns out that $\mathcal{A}^{\infty}(\Lambda, c)=\bigcap_{s\ge 0} \mathcal{A}^1_s (\Lambda, c) $ is equal to  the smooth noncommutative torus $A$.

 A dual lattice to $\Lambda$ is defined by
\[ \Lambda^{\circ}=\{ z\in \mathbb{R}^2 \mid \pi(\lambda)\pi(z)=\pi(z)\pi(\lambda) \,\, \text{for all $\lambda \in \Lambda$}\}.\]
Then we have $C^*(\Lambda^\circ, \bar{c}), \mathcal{A}^1_s (\Lambda^{\circ}, \bar{c}), \mathcal{A}^{\infty}(\Lambda^{\circ}, \bar{c})$ similarly. Let $M_s^1(\mathbb{R})$ be the modulation space in time-frequency analysis. More explicitly,
\[M_s^1(\mathbb{R})= \{ f \in L^2(\mathbb{R}) \mid \| f \|_{M^1_s}:= \int_{\mathbb{R}} |V_{\phi} f (x, \omega)| (1+|x|^2+|\omega|^2)^{s/2} dx d\omega < \infty \}\] where $V_{\phi}f$ is the \emph{short-time Fourier transform} of a function $f$ with respect to the window $\phi$, which is defined by $\lb f, \pi(x,\omega) \phi \rb_{L^2(\mathbb{R})}$. We can characterize the Schwartz space in terms of modulation spaces:
 \[ \mathscr{S}(\mathbb{R})=\bigcap_{s\ge 0} M^1_s(\mathbb{R}).\]
 
\begin{thm}\cite[Theorem 2.3]{L:Gabor}
For any $s\ge 0$ $M_s^1(\mathbb{R})$ is an equivalence bimodule between $\mathcal{A}^1_s(\Lambda, c)$ and $\mathcal{A}^1_s(\Lambda^{\circ}, \bar{c})$ and $\mathscr{S}(\mathbb{R})$ is an equivalence bimodule between $\mathcal{A}^{\infty}(\Lambda, c)$ and $\mathcal{A}^{\infty}(\Lambda^{\circ}, \bar{c})$.
\end{thm}
Although we do not need here, the strong Morita equivalence between $C\sp*(\Lambda, c)$ and $C\sp*(\Lambda^{\circ}, \bar{c})$ can be obtained using the above theorem.  From now on, $B$ denotes  $\mathcal{A}^{\infty}(\Lambda^{\circ}, \bar{c})$, which is also a smooth noncommutative torus for $-1/\theta$ \cite{R:NonTorus}. We note that for $f,g\in \mathscr{S}(\mathbb{R})$ 
\[{}_A\lb f, g\rb= \sum_{\lambda}\lb f,\pi(\lambda)g\rb \pi(\lambda),\]
\[\lb f,g \rb_B =\sum_{\lambda^{\circ}}\lb \pi(\lambda^{\circ})g, f \rb \pi^*(\lambda^{\circ}),\]
\[\pi(\mathbf{a})\cdot f=\sum_{\lambda}a(\lambda)\pi(\lambda)f \quad \text{for $\mathbf{a}\in l^1(\Lambda)$},\]
\[f\cdot \pi(\mathbf{b})=vol(\Lambda)^{-1}\sum_{\lambda^{\circ}}b(\lambda^{\circ})\pi^*(\lambda^{\circ})f\quad \text{for $\mathbf{b}\in l^1(\Lambda^{\circ})$}.\]
 The following theorem shows that Rieffel-type projections are linked to a hard problem in Gabor analysis.
 \begin{thm}
Suppose  $g \in \mathscr{S}(\mathbb{R})$. Then $\mathcal{G}(g, \Lambda)$ is a Gabor frame if and only if $\lb g, g \rb_B$ is invertible.
\end{thm}
\begin{proof}
 Given an equivalence bimodule $\Xi$ between $A$ and  $B$, we denote by $\End_A(\Xi)$ the algebra of module endomorphisms with respect to the action of $A$ on $\Xi$. It is well known that the equivalence between $\End_A(\Xi)$ and $B$ via $b \mapsto \phi_b$ where $\phi_b(\xi)=\xi\cdot b$ for $\xi\in \Xi$. We note that $\mathcal{G}(g, \Lambda)$ is a Gabor frame when $\Theta_{g,g}^{\Lambda}\in \End_A(\mathscr{S}(\mathbb{R}))$ is invertible where $\Theta_{g,g}^{\Lambda}(f)=f\cdot \lb g,g \rb_B$ for $f\in  \mathscr{S}(\mathbb{R})$ since $\lb\Theta_{g,g}^{\Lambda}(f),f\rb=\sum_{\lambda}\lb f,\pi(\lambda)g \rb \lb \pi(\lambda)g, f\rb=\sum_{\lambda}|\lb f, \pi(\lambda)g\rb|^2$ for $f\in \mathscr{S}(\mathbb{R})$.  Thus  the invertibility of $\Theta_{g,g}^{\Lambda}$ implies that the invertibility of $\lb g, g\rb_B$, and vice versa. 
\end{proof}
\begin{thm}\cite[Theorem 3.3]{L:Gabor}
Let $\mathcal{G}(g, \Lambda)$ be a Gabor system on $L^2(\mathbb{R})$ with $g$ in $\mathscr{S}(\mathbb{R})$. Then the following conditions are equivalent.
\begin{item}
\item[(i)] $\mathcal{G}(g, \Lambda)$ is a tight Gabor frame for  $L^2(\mathbb{R})$.
\item[(ii)] $\mathcal{G}(g, \Lambda^{\circ})$ is an orthogonal system.
\item[(iii)]$ \lb g, g \rb_B= 1_B$.
\item[(iv)]$ \lb g, \pi(\lambda^{\circ})g\rb_{L^2(\mathbb{R})}=vol(\Lambda) \delta_{\lambda^{\circ}, 0}$ for all $\lambda^{\circ}\in \Lambda^{\circ}$.
\end{item}
\end{thm}
An important fact for us is that  if $\mathcal{G}(g, \Lambda)$ is a Gabor frame for $L^2(\mathbb{R})$, then $\{\pi(\lambda^{\circ})g \mid \lambda^{\circ} \in \Lambda^{\circ}\}$ is a Riesz basis for the closed linear span of the set $\{\pi(\lambda^{\circ})g \mid \lambda^{\circ} \in \Lambda^{\circ}\}$. Moreover, if we take $\widetilde{g}=g \lb g,g\rb_B^{-1/2}$, then $\mathcal{G}(\widetilde{g}, \Lambda)$ becomes a tight Gabor frame.  We interpret  Wexler-Raz duality in Gabor analysis in terms of module relations as it appeared in \cite{DLF:Sigma} without a proof.   
\begin{thm}
\begin{equation}\label{E:Wexler-Raz}
f=\widetilde{g}\lb \widetilde{g}, f \rb_B
\end{equation} for $f\in \mathscr{S}(\mathbb{R})$.
\end{thm}
\begin{proof}
Since the system $\{\pi(\lambda^{\circ})\widetilde{g}\mid \lambda^{\circ}\in \Lambda^{\circ}\}$ is dual to itself, \cite[Theorem 1.2.2 p.40]{FS} implies that
\[ \lb f, h\rb=\sum_{\lambda^{\circ}} \lb f, \pi(\lambda^{\circ})\widetilde{g} \rb \lb \pi(\lambda^{\circ})\widetilde{g}, h\rb, \quad f,h \in L^2(\mathbb{R}). \]
Since $\pi(\lambda^{\circ})^*=\overline{c(\lambda^{\circ}, \lambda^{\circ})}\pi(-\lambda^{\circ})$ and $|c(\lambda^{\circ}, \lambda^{\circ})|^2=1$,
\[ \begin{split}
\sum_{\lambda^{\circ}} \lb f, \pi(\lambda^{\circ})\widetilde{g} \rb \lb \pi(\lambda^{\circ})\widetilde{g},h\rb&=\sum_{\lambda^{\circ}} |c(\lambda^{\circ}, \lambda^{\circ})|^2\lb f, \pi(-\lambda^{\circ})\widetilde{g} \rb \lb \pi(-\lambda^{\circ})\widetilde{g},h\rb\\
&=\sum_{\lambda^{\circ}} \lb f, \overline{c(\lambda^{\circ},\lambda^{\circ})}\pi(-\lambda^{\circ})\widetilde{g} \rb \lb \overline{c(\lambda^{\circ},\lambda^{\circ})}\pi(-\lambda^{\circ})\widetilde{g},h\rb\\
&=\sum_{\lambda^{\circ}} \lb f, \pi^*(\lambda^{\circ})\widetilde{g} \rb \lb \pi^*(\lambda^{\circ})\widetilde{g},h\rb\\
&=\sum_{\lambda^{\circ}}\lb \pi(\lambda^{\circ})f, \widetilde{g}\rb \lb \widetilde{g}, \pi(\lambda^{\circ})h \rb\\
&=\lb \sum_{\lambda^{\circ}} \lb \pi(\lambda^{\circ})f, \widetilde{g} \rb\pi^*(\lambda^{\circ})\widetilde{g}, h \rb\\
&=\lb \widetilde{g}\cdot\lb \widetilde{g}, f \rb_B, h\rb.  
\end{split}\]
\end{proof}
Recently a large class of functions are proven to be Gabor frames \cite{GL:Gabor}. Recall that a function $\eta$ is \emph{totally positive} if for every two sets of increasing real numbers $x_1< \cdots < x_N$ and $y_1 < \cdots < y_N$ the determinant of the matrix $ \left(\eta(x_j - y_k)_{1\le j,k \le N} \right)$ is non-negative. A totally positive function $\eta$ is of \emph{finite type} $M$, $M\in \mathbb{N}$ with $M\ge 2$, if its Laplace transform $\widehat{\eta}$ has the form:
\[ \widehat{\eta}(\omega)=e^{-\delta\omega^2}e^{-\delta_0\omega}\prod_{j=1}^{M} \frac{1}{1+2\pi i \delta_j \omega} \] for real non-zero parameters $\delta_j, \delta >0$.
\begin{cor}\cite[Lemma 6.2]{DLF:Sigma}
Let $\eta$ be a totally positive function of finite type greater than 2. Then $\eta$ is a standard module frame for $ \mathscr{S}(\mathbb{R})$. Passing to a Parseval frame $\tilde{\eta}=\eta \cdot \lb \eta, \eta \rb_B^{-1/2}$, $p_{\tilde{\eta}}={}_{A}\lb \tilde{\eta}, \tilde{\eta} \rb $ is a Rieffel-type projection in $A$ for $0<\theta<1$.
\end{cor}
We remark that such examples include  Gaussians, and hyperbolic secants. \\

\section{A group action on solitons}\label{S:Solitons}

A commutative torus $\mathbb{T}^2$-action $\alpha$ is defined on $A_{\theta}$ by
\[\alpha_{(z_1, z_2)}(U_1)=z_1U_1, \alpha_{(z_1, z_2)}(U_2)=z_2U_2 \quad \text{for} \,(z_1,z_2)\in \mathbb{T}^2.\]

We denote by $\partial_1$ and $\partial_2$ the infinitesimal generators of each factor of $\mathbb{T}^2$ under $\alpha$ \cite{CO:NCG}.  These are unbounded derivations on $A_{\theta}$, but well defined on $\mc{A_{\theta}}$. For $\nu, \mu =1,2$
\[ \partial_{\nu}(U_{\mu})=2\pi i \delta_{\nu, \mu} U_{\mu}. \]

Similarly, we have such derivations on $B$, and use same notations without confusion.

Equipped with a $B$-valued hermitian structure $\lb \,, \, \rb_B $ on $\Xi=\mathscr{S}(\mathbb{R})$, we  can lift derivations to covariant derivatives $\nabla_1, \nabla_2$ on $\Xi$ given by
\[ (\nabla_1 \xi)(t)=\frac{2\pi i t}{\theta} \xi(t) \quad \text{and} \quad (\nabla_2 \xi)(t)= \xi ' (t). \]
Then as proved in \cite {CR:YM} we have the (right) Leibnitz rule for both covariant derivatives:
\begin{equation}\label{E:Leibniz}
\nabla_{\nu}(\xi\cdot b)=(\nabla_{\nu}\xi)\cdot b+ \xi\cdot (\partial_{\nu} b)  \quad \text{for $\nu =1,2 $,}
\end{equation}
and compatibility with the hermitian structure:
\begin{equation}\label{E:compatibility}
\partial_{\nu}(\lb \xi_1, \xi_2 \rb_B)=\lb \nabla_{\nu}\xi_1, \xi_2 \rb_B+ \lb \xi_1, \nabla_{\nu}\xi_2 \rb_B \quad \text{for $\nu =1,2 $.}
\end{equation}
We introduce complex derivations $\overline{\partial}= \partial_1+ i \partial_2$ and $\partial=\partial_1- i \partial_2$. Accordingly, we introduce the anti-holomorphic connection $\overline{\nabla}=\nabla_1 + i  \nabla_2$ and the holomorphic one $\nabla=\partial_1-i\partial_2$. Then using linearity (\ref{E:Leibniz}) and (\ref{E:compatibility}) hold for $\overline{\nabla}(\nabla)$ and $\overline{\partial}(\partial)$ respectively.
\begin{lem}\label{L:Killing}
Let $p$ be a projection in $\mc{A_{\theta}}$. Then $p(\partial_{\nu}p)p=0$.
\end{lem}
\begin{proof}
Note that
\[ \partial_{\nu}(p)=(\partial_{\nu}p)p +p(\partial_{\nu}p).\]
By multiplying $p$ both sides, the conclusion follows.
\end{proof}

\begin{lem}\label{L:Action}
$\displaystyle \Tr(\partial_{\nu}p \partial_{\nu}p)=2 \Tr(p\partial_{\nu}p \partial_{\nu}p)$
\end{lem}
\begin{proof}
\[
\begin{split}
\Tr(\partial_{\nu}p \partial_{\nu}p)=& \Tr( [(\partial_{\nu}p)p +p(\partial_{\nu}p)][(\partial_{\nu}p)p +p(\partial_{\nu}p)])\\
=& \Tr((\partial_{\nu}p)p(\partial_{\nu}p)p+p(\partial_{\nu}p \partial_{\nu}p) p+(\partial_{\nu}p)p(\partial_{\nu}p)+ (\partial_{\nu}p) p (\partial_{\nu}p)  p)\\
=&2 \Tr(p\partial_{\nu}p \partial_{\nu}p)
\end{split}
\]
using Lemma \ref{L:Killing} and cyclicity of $\Tr$.
\end{proof}
Let us recall a characterization of the minimizing solitons  from \cite{DKL:Sigma,DKL:Sigma2}; let $Q(p)$  be the topological charge or the first Chern number defined by
\[ \frac{1}{2\pi i}\Tr (p[\partial_1 p \partial_2 p-\partial_2 p\partial_1 p]), \] which is an integer \cite{CO:NCG}.
Since $\displaystyle \Tr ((\overline{\partial}(p)p)^*(\overline{\partial}(p)p))\geq 0$ or $\displaystyle \Tr((\partial(p)p)^* (\partial(p)p)) \ge 0$, we have,  combining (\ref{E:Action}) with Lemma \ref{L:Action},
\[S(p) \ge \pm 4\pi Q(p)\] where the equality occurs exactly when the self duality equation\begin{equation*}
\overline{\partial}(p)p=0
\end{equation*}  holds or the anti-self duality equation  \begin{equation*}
\partial(p)p=0
\end{equation*} holds since $\Tr$ is faithful.\\

An important result of \cite{DKL:Sigma} is the following observation which lifts a self-duality equation to a linear equation involving the anti-holomorphic connection on the module.
\begin{thm}\cite[Section 5.3]{DKL:Sigma2}\label{T:DLL}
Let $\xi$ be a standard Parseval frame or $\lb \xi, \xi \rb_B=1_B$. Let $p_{\xi}={}_{A}\lb \xi, \xi \rb$ be a projection. Then
\[(\overline{\partial}p_{\xi}) p_{\xi}=0  \quad \text{if and only if} \quad \overline{\nabla}\xi=\xi \cdot b  \quad   \text{for some $b \in B $}\]
\end{thm}
\begin{rem}
In fact, if $\overline{\nabla}\xi=\xi \cdot b$ holds, then $b$ must be $\lb \xi, \overline{\nabla}\xi \rb_B$.
\end{rem}
Whenever we have a standard module frame $\eta$, then by passing from a standard module frame to a Parseval one $\tilde{\eta}$ one gets a noncommutative soliton.  With Theorem \ref{T:Trace} in mind,  we need a slightly stronger form of a linear equation of $\eta$ than $\tilde{\eta}$. From now on, we write $A$ instead of $\mc{A_{\theta}}$ for the subscript in the operator valued inner product.   

\begin{thm}\label{T:Soliton}
 Suppose that $\langle \eta, \eta \rangle_B$ is invertible. Then let $\tilde{\eta}=\eta \langle \eta, \eta \rangle_B ^{-1/2}$ and $p_{\tilde{\eta}}=_A\lb \tilde{\eta}, \tilde{\eta} \rb$.  Then
\[(\overline{\partial}p_{\tilde{\eta}}) p_{\tilde{\eta}}=0  \quad \text{if and only if} \quad \overline{\nabla}\eta=\eta \cdot b  \quad   \text{for some $b \in B $} \]
\end{thm}
\begin{proof}
In view of Theorem \ref{T:DLL}, the easier way is to show that the equivalence between $\overline{\nabla}\tilde{\eta}=\tilde{\eta}\cdot b'$ for some $b' \in B$ and $\overline{\nabla}\eta=\eta \cdot b $  for some $b\in B$.\\
 In the following we give a direct proof.
For this, it is better to view $p_{\tilde{\eta}}$ as ${}_{A} \lb \eta \cdot \lb \eta, \eta \rb_B^{-1}, \eta \rb $ or ${}_A\lb \eta, \eta\cdot \lb \eta, \eta \rb_B^{-1} \rb$. We abbreviate $p_{\tilde{\eta}}$ as $p$ without confusion. First note the following cancellation property; \begin{equation}\label{E:cancel}
{}_{A}\lb \zeta, \eta\rb p={}_{A}\lb \zeta, \eta\rb \quad \forall \zeta.
\end{equation}
Then
\[
\begin{aligned}
\partial_{\nu}(p)p=&{}_{A}\lb\nabla_{\nu}(\eta \cdot \lb \eta, \eta \rb_B^{-1} ), \eta \rb p +{}_{A}\lb \eta \cdot \lb \eta, \eta\rb_B^{-1}, \nabla_{\nu}\eta\rb p\\
=& {}_{A}\lb \nabla_{\nu}(\eta \cdot \lb \eta, \eta \rb_B^{-1} ), \eta \rb + {}_{A}\lb \eta\lb \eta, \eta\rb_B^{-1}\lb \nabla_{\nu}\eta, \eta \rb_B \lb \eta, \eta \rb_B^{-1}, \eta \rb.
\end{aligned}
\]
Since \[\nabla_{\nu}(\eta \cdot \lb \eta, \eta \rb_B^{-1})=(\nabla_{\nu}\eta)\cdot \lb \eta, \eta \rb_{B}^{-1}+ \eta \partial_{\nu}(\lb \eta, \eta \rb_B^{-1}),\] and 
\[\partial_{\nu}(\lb \eta, \eta \rb_B^{-1})=-\lb \eta, \eta \rb_B^{-1} \left (\lb \nabla_{\nu}\eta , \eta \rb_B + \lb \eta, \nabla_{\nu}\eta \rb_B \right) \lb \eta, \eta \rb_B^{-1}, \] 

\[
\begin{aligned}
\partial_{\nu}(p)p=& {}_{A}\lb \nabla_{\nu}(\eta \cdot \lb \eta, \eta \rb_B^{-1} ), \eta \rb + {}_{A}\lb \eta\lb \eta, \eta\rb_B^{-1}\lb \nabla_{\nu}\eta, \eta \rb_B \lb \eta, \eta \rb_B^{-1}, \eta \rb\\
=& {}_{A}\lb \nabla_{\nu}\eta \cdot \lb \eta, \eta \rb_B^{-1}, \eta \rb - p {}_{A}\lb \nabla_{\nu}\eta \cdot \lb \eta, \eta \rb_B^{-1}, \eta \rb\\
=&(1 - p)  {}_{A}\lb \nabla_{\nu}\eta, \eta \cdot \lb \eta, \eta \rb_B^{-1} \rb.
\end{aligned}
\]
Using linearity, it follows that
\[\overline{\partial}(p) p=0 \quad \text{if and only if}\quad (1 - p)  {}_{A}\lb \overline{\nabla}\eta, \eta \cdot \lb \eta, \eta \rb_B^{-1} \rb=0.\]
Then if $ \overline{\nabla}\eta=\eta \cdot b$ for some $b\in B$,
\[
\begin{aligned}
(1 - p) {}_{A}\lb \overline{\nabla}_{\nu}\eta, \eta \cdot \lb \eta, \eta \rb_B^{-1} \rb=&{}_{A}\lb \eta\cdot b, \eta\cdot \lb \eta, \eta \rb_B^{-1} \rb- {}_{A} \lb \eta \cdot \lb \eta, \eta \rb^{-1} , \eta \rb {}_{A}\lb \eta \cdot b, \eta\cdot \lb \eta, \eta \rb_B^{-1} \rb\\
=&{}_{A}\lb \eta\cdot b, \eta\cdot \lb \eta, \eta \rb_B^{-1} \rb - {}_{A}\lb \lb \eta \cdot \lb \eta, \eta \rb_B^{-1} \lb \eta, \eta \cdot b \rb_B, \eta\cdot \lb \eta, \eta \rb_B^{-1} \rb\\
=&0.
\end{aligned}
\]
Conversely, if $\overline{\partial}(p) p=0 $, then $(1 - p)  {}_{A}\lb \overline{\nabla}\eta, \eta \cdot \lb \eta, \eta \rb_B^{-1} \rb=0$. It follows that
\[
\begin{split}
0&=((1 - p)  {}_{A}\lb \overline{\nabla}\eta, \eta \cdot \lb \eta, \eta \rb_B^{-1} \rb )\cdot \eta\\
&=(1-p) ({}_{A}\lb \overline{\nabla}\eta, \eta \cdot \lb \eta, \eta \rb_B^{-1} \rb\cdot \eta)\\
&=(1-p) (\overline{\nabla}\eta \cdot \lb \eta \cdot \lb \eta, \eta \rb_B^{-1}, \eta \rb_B)\\
&=(1-p)\overline{\nabla}\eta.
\end{split}
\]
Therefore,
\[
\begin{split}
\overline{\nabla}\eta&=p\overline{\nabla} \eta \\
&={}_{A}\lb \eta \cdot \lb \eta, \eta \rb_B^{-1}, \eta \rb \overline{\nabla}\eta\\
&=\eta \cdot (\lb \eta, \eta \rb_B^{-1}\lb \eta, \overline{\nabla}\eta \rb_B)
\end{split}
\]
\end{proof}
\begin{rem}
If $\overline{\nabla}\eta=\eta\cdot b$, then $b$ must be $\lb \eta, \eta \rb_B^{-1}\lb \eta, \overline{\nabla}\eta \rb_B$; note that $\overline{\partial}(\lb \eta, \eta \rb_B \lb \eta, \eta \rb_B^{-1})=0=\overline{\partial}(\lb \eta, \eta \rb_B^{-1} \lb \eta, \eta \rb_B)$. From the  first equality,
\begin{equation}
\begin{split}
(\lb \overline{\nabla}\eta, \eta\rb_B +\lb \eta, \overline{\nabla}\eta \rb_B) \lb \eta, \eta \rb_B^{-1}+ \lb \eta, \eta \rb_B \overline{\partial}(\lb \eta, \eta \rb_B^{-1}) =0.
\end{split}
\end{equation}
Thus we have
\begin{equation}\label{E:dbar}
\overline{\partial}(\lb \eta, \eta \rb_B^{-1} )= -\lb \eta, \eta \rb_B^{-1} (\lb \overline{\nabla}\eta, \eta\rb_B +\lb \eta, \overline{\nabla}\eta \rb_B)\lb \eta, \eta \rb_B^{-1}.
\end{equation}
From the second equality,
\begin{equation}\label{E:uniqueness}
\overline{\partial}(\lb \eta, \eta \rb_B^{-1} )\lb \eta, \eta \rb_B + \lb \eta, \eta \rb_B^{-1}(\lb \overline{\nabla}\eta, \eta\rb_B +\lb \eta, \overline{\nabla}\eta \rb_B) =0.
\end{equation}
In (\ref{E:uniqueness}), substitute $\overline{\partial}(\lb \eta, \eta \rb_B^{-1} )$ using (\ref{E:dbar}) and $\overline{\nabla}\eta$ by $\eta \cdot b$ in the last term
\[
\begin{split}
&-\lb \eta, \eta \rb_B^{-1} (\lb \overline{\nabla}\eta, \eta\rb_B +\lb \eta, \overline{\nabla}\eta \rb_B)\lb \eta, \eta \rb_B^{-1} \lb \eta, \eta \rb_B + \lb \eta, \eta \rb_B^{-1}( \lb \overline{\nabla}\eta, \eta \rb_B+ \lb \eta, \eta\cdot b \rb_B )\\
&=-\lb \eta, \eta \rb_B^{-1} \lb \eta, \overline{\nabla}\eta \rb_B +b=0.
\end{split}
\]
\end{rem}

\begin{cor}\cite[Proposition 6.3]{DLF:Sigma}
Let $\eta$ be a Gabor frame in $\Xi$. Then $p_{\tilde{\eta}}$ is a solution of the self  duality equation.
\end{cor}
\begin{proof}
By  Wexler-Raz duality (\ref{E:Wexler-Raz}) for a tight Gabor frame, we have
\[\zeta= \tilde{\eta}\cdot \lb \tilde{\eta}, \zeta \rb_B \quad \forall \zeta \in \Xi.  \]
Since $\eta$ in $\Xi$, so is $\overline{\nabla}(\eta)$. Thus taking $\zeta$ as $\overline{\nabla}\eta$
\[\overline{\nabla}(\eta)=\eta \dot \lb \eta, \eta \rb_B^{-1}\lb \eta,\overline{\nabla} \eta \rb_B. \]
Then the conclusion follows from Theorem \ref{T:Soliton}.
\end{proof}

Let $\GL(B)$ be the set of invertible elements in $B$. An action of $\GL(B)$ on noncommutative solitons was introduced in \cite{DKL:Sigma} by the right multiplication; for $U\in GL(B)$
\[ \xi \to \xi \cdot U=\xi_U.\]
Indeed, if $\xi$ satisfies a self-duality equation, or $\overline{\nabla}(\xi)=\xi \cdot b$ for some $b\in B$, then  by the Leibniz rule for the connection, one finds that $\xi_U$ is the solution of an equation of the form $\overline{\nabla}\xi_U=\xi_U \cdot b_U$ where 
\begin{equation}\label{E:Gaugetransformed}
b_U=U^{-1}b U+U^{-1}\overline{\partial}U.
\end{equation}
Note that this action preserves the invertibility of $\lb \xi, \xi \rb_B$, thus preserves Gabor frames. Moreover, it is not difficult to check that the Rieffel-type projections are invariant under the action (see \cite[p. 232]{L:Harmonic}). 

When $\lambda$ is a scalar, i.e., $\lambda\in \mathbb{C}$, $\overline{\nabla}\eta=\eta \cdot \lambda$ has the solutions,  the Gaussians of the form $Ce^{-\pi \theta t^2-2i \lambda t}$.  In \cite{DKL:Sigma, DKL:Sigma2} it is analyzed  when two Gaussian solitons are gauge to each other.
\begin{prop}
Let $\xi$ be a solution of equation with $\lambda \in \mathbb{C}$; and let $U\in\GL(B)$. Then the transformed $\lambda_U$ will be again constant if and only if there exists a pair of integers $(m,n)$ such that 
\begin{equation*}
U=C_{mn}U_1^mU_2^n.
\end{equation*}
Furthermore, 
\begin{equation*}
\lambda_U-\lambda=\pi i(m+ni).
\end{equation*}
\end{prop}

In \cite{DKL:Sigma, DKL:Sigma2}  Dabrowski, Krajewski, and  Landi suggest the following question;\\
 
Q: is it possible to gauge a Gaussian soliton to any solution of the self duality equation?\\  

In view of Theorem \ref{T:Soliton}, this question is equivalent to the following statement: Choose a $\lambda \in \mathbb{C}$,  then for any $b \in B$ is there an element $U\in \GL(B)$ such that $b=\lambda+U^{-1}\overline{\partial}U$ ? This is related to solving \emph{inhomogeneous Cauchy-Riemann equation of the form}
 \begin{equation}\label{E:CauchyRiemann}
\overline{\partial}U=U(b-\lambda).
\end{equation}

Based on the following Polishchuk's observation, this question is reduced to compute the trace of $b$(see Corollary \ref{C:Gaugeaction}).

\begin{thm}\cite[Theorem 3.6]{P:Ex}, \cite[Theorem 6.2]{Ro:Laplace}
Let $\tau$ be the unique trace of the noncommutative torus $B$.  Then for $b \in B$, $U^{-1}\overline{\partial}U=b$ has a nontrivial solution if and only if $\tau(b)\in \pi i (\mathbb{Z}+i\mathbb{Z})$.
\end{thm}
\begin{proof}
This follows from (1) of \cite[Theorem 3.6]{P:Ex}. Because of a slightly different notation for $\overline{\partial}$ up to the factor $2$, the range of $\tau(U^{-1}\overline{\partial}U)$ is changed to $\pi i(\mathbb{Z}+i\mathbb{Z})$. 
\end{proof}

\begin{cor}\label{C:Gaugeaction}
Let $\eta (\in \Xi)$ be a solution of $\overline{\nabla}\eta=\eta \cdot b$ for some $b\in B$. Then there is a Gaussian $\xi$ and $U\in \GL(B)$ such that $\eta=\xi \cdot U$ if and only if $\lambda-\tau(b) \in \pi i (\mathbb{Z}+i\mathbb{Z})$ where $\overline{\nabla}\xi=\xi \cdot \lambda $ for $\lambda \in \mathbb{C}$.
\end{cor}

It is already mentioned in \cite{L:Harmonic} that the question Q is not true in general as we know a constraint exists. However, in some good cases it is possible to find a nontrivial solution of (\ref{E:CauchyRiemann}) for some $\lambda\in \mathbb{C}$. In other words, we can gauge a class of noncommutative solitons to  Gaussian solitons.  

\begin{thm}\label{T:Trace}
Let $\eta$ be a Gabor frame for Gabor system $\mc{G}(\eta, \Lambda)$ and  satisfy $\overline{\nabla}\eta=\eta \cdot b $ for some $b \in B$.  Then $\tau(b)=\lb \overline{\nabla}\eta , \eta \rb_{L^2(\mathbb{R})} $ where $\lb \cdot,\cdot \rb_{L^2(\mathbb{R})}$ is the inner product on the Hilbert space $L^{2}(\mathbb{R})$.
\end{thm}
\begin{proof}
By Wexler-Raz biorthogonality the fact that $\mc{G}(\eta, \Lambda)$ is a tight Gabor frame for $L^2(\mathbb{R})$ implies that $\mc{G}(\eta, \Lambda^{\circ})$ is an orthogonal system, i.e. a Riesz basis for $L^2(\mathbb{R})$. Therefore we can expand $\overline{\nabla}\eta$ in terms of $\{\pi(\lambda^{\circ})\eta\mid \lambda^{\circ}\in \Lambda^{\circ}\}$. Write $b=\sum_{\lambda^{\circ} \in \Lambda^{\circ}}b(\lambda^{\circ})\pi(\lambda^{\circ})$ where $b(\lambda^{\circ})$'s are rapidly decreasing.
Then\begin{equation*} \begin{aligned} \eta\cdot b=&\sum_{\lambda^{\circ}\in \Lambda^{\circ}}b(\lambda^{\circ})\pi(\lambda^{\circ})^*\eta\\
=&\sum_{\lambda^{\circ}\in \Lambda^{\circ}}b(-\lambda^{\circ})c(\lambda^{\circ},\lambda^{\circ})\pi(\lambda^{\circ})\eta
\end{aligned}
\end{equation*}
Thus the condition $\overline{\nabla}\eta=\eta \cdot b $ implies that
\[ \sum_{\lambda^{\circ}\in \Lambda^{\circ}}\lb \overline{\nabla}\eta, \pi(\lambda^{\circ})\eta \rb_{L^2(\mathbb{R})}\pi(\lambda^{\circ})\eta= \sum_{\lambda^{\circ}\in \Lambda^{\circ}}b(-\lambda^{\circ})c(\lambda^{\circ},\lambda^{\circ})\pi(\lambda^{\circ})\eta.\] It follows that  from the orthogonality of $\{\pi(\lambda^{\circ})\eta\mid \lambda^{\circ}\in \Lambda^{\circ}\}$
\[
\lb \overline{\nabla}\eta, \pi(\lambda^{\circ})\eta \rb_{L^2(\mathbb{R})}= b(-\lambda^{\circ})c(\lambda^{\circ},\lambda^{\circ}).
\]
Hence $\tau(b)=b(0,0)=\lb \overline{\nabla}\eta , \eta \rb_{L^2(\mb{R})}$.
\end{proof}
\begin{prop}\label{P:HypersecantGaussian}
Let $\eta$ be the hyperbolic secant of the form  $\displaystyle \left(\frac{\pi}{2}\right)^{\frac{1}{2}}\frac{1}
 {\cosh(\pi t)}$. Then $\tau(b)=0$ where $\overline{\nabla}\eta=\eta \cdot b$, so that the Gaussians $\xi$ associated with $\overline{\nabla}\xi=\xi \cdot \lambda$ for $\lambda \in \pi i (\mathbb{Z}+i\mathbb{Z})$  are gauged to $\eta$. Moreover, if the invertible $W$ such that $\eta=\xi \cdot W$  cannot be of the form $U^mV^n$ (modulo $\mathbb{T}$ ) where $U$ and $V$ are generators of $B$.
\end{prop}
\begin{proof}
 $\eta$  is a Gabor frame due to Janssen and Strohmer and $p_{\eta}$ belongs to $B$ by \cite[Theorem 3.6]{L:Gabor}.
Note that $\displaystyle \overline{\nabla}\eta(t)=i\pi\left(\frac{2t}{\theta}-\tanh(t)\right)\eta(t)$  up to a constant. Therefore
\[\lb \overline{\nabla}\eta , \eta \rb_{L^2(\mathbb{R})}= 2\pi i\int_{-\infty}^{\infty} t \eta^2(t)dt -i \pi\int_{-\infty}^{\infty} \tanh(t)\eta^2(t)dt. \] Since $\eta$ is an even function and both $t$ and $\tanh(t)$ are odd functions, two terms vanish by the definition of the Lebesgue integral. Thus $\tau(b)=0$. Since $\eta$ is a Gabor frame in $\Xi$, then it satisfies $\overline{\nabla}\eta=\eta \cdot b$ for some $b$. Moreover, if $\eta=\xi \cdot W$, then $\tau(b)=\lambda+\tau(W^{-1}\overline{\partial}W)$ by (\ref{E:Gaugetransformed}). Thus  we must have $\lambda=-\tau(W^{-1}\overline{\partial}W) \in  \pi i (\mathbb{Z}+i\mathbb{Z})$. The last statement follows from the fact that two Gaussians are gauge equivalent if and only if one is gauged to the other via $U^mV^n$(for some $m,n$) only up to constants \cite{DKL:Sigma}.
\begin{rem}
The second statement in Proposition \ref{P:HypersecantGaussian} is related to the author's question  in \cite{Lee:Sigma} if any two solutions of (\ref{E:EL}) are equivalent under a $\mathbb{Z}^2$-action which is defined by an inner automorphisms $\Ad W$ where $W$ is a unitary of the form $U^mV^n$ modulo $\mathbb{T}$. The answer is no as we see from the above example. 
\end{rem}
\end{proof}

\section{Acknowledgements}
The author would like to thank F. Luef for a series of lectures of his work during his visit to Korea. He also would like to express his gratitude to  Hun Hee Lee for a final tip in proving Theorem \ref{T:Trace}. I


\begin{bibdiv}

\begin{biblist}
\bib{CO:NCG}{article}{
   author={Connes, A.},
   title={$C\sp*$-alg\`{e}bres et g\'{e}ometrie diff\'{e}rentille},
   journal={C.R. Acad. Sci. Paris S\'{e}r. A},
   volume={290},
   number={13}
   date={1980},
   pages={599--604},
   issn={},
   review={MR1690050(81c:46053)},
   doi={},
}
\bib{B:QT}{article}{
   author={Boca, F.},
   title={Projections in rotation algebras and theta functions},
   journal={Comm. Math. Phys.},
   volume={202},
   date={1999},
   number={2}
   pages={325--357},
   issn={},
   review={MR1690050(2000j:46101)},
   doi={},
}
\bib{CR:YM}{article}{
   author={Connes, A.},
   author={Rieffel, M.}
   title={ Yang-Mills for noncommutative two-tori},
   journal={Contemp. Math.},
   volume={62},
   date={1987},
   pages={335--348},
   issn={0075-4102},
   review={\MR{454645 (56\#:12894)}},
   doi={},
}

\bib{DKL:Sigma}{article}{
   author={Dabrowski, L.},
   author={Krajewski, T.},
   author={Landi, G.},
   title={Some properties of Non-linear $\sigma$-models in noncommutative geometry},
   journal={Int. J. Mod. Phys.},
   volume={B14},
   date={2000},
   pages={2367--2382},
   review={\MR{0470685 (57 \#10431)}},
}

\bib{DKL:Sigma2}{article}{
   author={Dabrowski, L.},
   author={Krajewski, T.},
   author={Landi, G.},
   title={Non-linear $\sigma$-models in noncommutative geometry: fields with values in finite spaces},
   journal={Mod. Physics Lett. A},
   volume={18},
   date={2003},
   pages={2371--2379},
   review={},
}


\bib{DLF:Sigma}{article}{
   author={Dabrowski, L},
   author={Landi, G.},
   author={Luef, F}
   title={Sigma-model solitons on noncommutative spaces},
   journal={Lett. Math. Phys.},
   volume={105},
   date={2015},
   number={12},
   pages={1633--1688},
   issn={},
   review={\MR{3420593}},
   doi={10.1007/s11005-015-0790-x}
}

\bib{EE:Irrational}{article}{
  author={Elliott, G.},
  author={Evans, D.}, 
  title={The structure of the irrational rotation $C\sp*$-algebra}, 
  journal={Ann. of Math.},
   volume={(2)138},
  date={1993},  
number={3}, 
pages={477–-501}, 
review={\MR{1247990 (94j:46066)}}
}

\bib{FS}{book}{
author={Feichtinger, H},
author={Strohmer},
title={Gabor analysis and Algorithms},
publisher={Springer Science+Business Media, LLC}
date={1998},
}
\bib{FL:Frame}{article}{
   author={Frank, M.},
   author={Larson, D.},
   title={Frames in Hilbert $C\sp*$-modules and $C\sp*$-algebras},
   journal={J. Operator Theory},
   volume={},
   date={2002},
   number={},
   pages={273--314},
   issn={},
   doi={},
}
\bib{GL:Gabor}{article}{
   author={Gr\"{o}chenig, K.},
   author={Lyubarskii, Y.}
   title={Gabor (super)frames and totally positive functions},
   journal={Duke  Math. J},
   volume={162},
   date={2013},
   number={5},
   pages={1003--1031},
   issn={},
   review={},
   doi={},
}

\bib{L:Harmonic}{article}
{author={Landi, G.}
title={On harmonic maps in noncommutative geometry},
 journal={ Non-commutative Geometry and Number Theory \,Springer}, 
 pages={217--234},
 date={2006},
 }

\bib{Lee:Sigma}{article}{
   author={Lee, H.},
   title={A note on nonlinear $\sigma$-models in noncommutative geometry},
   journal={IDAQP},
   volume={19},
   date={2016},
   number={1},
   pages={},
   issn={0022-1236},
   review={\MR{2733573 (2011k:46079)}},
   doi={10.1142/S0239025716500065},
}
\bib{L:Gabor}{article}{
   author={F. Luef},
   title={Projections in noncommutative tori and Gabor frames},
   journal={Proc. A.M.S.},
   volume={139},
   date={2010},
   number={2},
   pages={571--582},
   issn={0002-9939},
   review={},
   doi={},
}
\bib{L:VBoverNT}{article}{
   author={Luef, F.},
   title={Projective modules over noncommutative tori are multi-window Gabor frames for modulation spaces},
   journal={J. Funct. Anal.},
   volume={257},
   date={2009},
   number={6},
   pages={1921--1946},
   issn={0379-4024},
   review={\MR{2540994}},
}
\bib{MR:Sigma}{article}{
   author={Mathai, V.},
  author={Rosenberg, J.}
   title={A noncommutative sigma-model},
   journal={J. Noncommut. Geom.},
   volume={5},
   date={2011},
   number={},
   pages={265--294},
   issn={},
   review={},
   doi={},
}
\bib{P:Ex}{article}{
   author={Polishchuck, P.},
   title={Analogues of the exponential map associated with complex structures on noncommutative two-tori},
   journal={Pacific J. Math.},
   volume={226},
   date={2006},
   number={1},
   pages={153--178},
   issn={},
   review={},
   doi={},
}
\bib{R:NonTorus}{article}{
   author={Rieffel, M.},
   title={$C\sp*$-algebras associated with irrational rotations},
   journal={Pacific. J. Math.},
   volume={93},
   date={1981},
   number={1},
   pages={415--429},
   review={\MR{623572(83b:46087)}},
   doi={},
   }
 \bib{Ro:Laplace}{article}{
   author={Rosenberg, J.},
   title={Noncommutative variations on Laplace equation},
   journal={Anal. PDE.},
   volume={1},
   date={2008},
   number={1},
   pages={   },
   review={\MR{2444094 (2009f:58041)}},
   doi={},
   }
\end{biblist}

\end{bibdiv}
\end{document}